\documentclass[a4paper,reqno]{amsart}
\usepackage{amssymb, amsmath, amscd}
\newtheorem{theorem}{Theorem}[section]
\newtheorem{definition}[theorem]{Definition}
\newtheorem{lemma}{Lemma}[section]
\newtheorem{remark}{Remark}[section]

\makeatletter
 \@addtoreset{equation}{section}
\makeatother
\begin{document}
\title{Universal curvature identities III}
\author{P. Gilkey, J.H. Park, and K. Sekigawa}
\address{PG: Mathematics Department, University of Oregon, Eugene OR 97403 USA}
\email{gilkey@uoregon.edu}
\address{JHP: Department of Mathematics, Sungkyunkwan University, Suwon
440-746, Korea \\\phantom{JHP:..A} Korea Institute for Advanced Study, Seoul
130-722, Korea} \email{parkj@skku.edu}
\address{KS: Department of Mathematics, Niigata University, Niigata, Japan.}
\email{sekigawa@math.sc.niigata-u.ac.jp}
\begin{abstract}{We examine universal curvature identities
for pseudo-Riemannian manifolds with boundary. We determine the
Euler-Lagrange equations associated to the Chern-Gauss-Bonnet
formula and show that they are given solely in terms of curvature
{and the second fundamental form and do not involve covariant
derivatives  thus generalizing a conjecture of Berger to this
context.}
}\end{abstract}

\maketitle

\section{Introduction}

The study of curvature is central to modern differential geometry and mathematical physics. Properties of the curvature operator have been examined
by many authors -- see, for example, the discussion in \cite{BGNS08, GHVV08}. Eta Einstein geometry has been investigated \cite{CPS09,P12}.
Curvature plays an important role in spectral geometry -- see,
for example, \cite{A08}. The Lorentzian and higher signature settings are of special importance \cite{CD09,MP09}.

The Gauss Bonnet theorem has many physical applications as are the associated Euler Lagrange Equations (see, for example, \cite{B07,I10,S09}).
This paper deals with universal curvature identities arising from the Euler-Lagrange equations for the Chern-Gauss-Bonnet
theorem for manifolds with boundary. In Section~\ref{sect-1.1} and in Section~\ref{sect-1.2}, we discuss
the Gauss-Bonnet theorem for closed pseudo-Riemannian
manifolds, present the associated
Euler-Lagrange equations, and discuss some of the historical context of the problem that we shall be considering.
In Section~\ref{sect-1.3},
 we recall the Gauss-Bonnet
theorem for manifolds with boundary; we shall always assume the restriction
of the pseudo-Riemannian metric to the boundary to be non-degenerate.
In Section~\ref{sect-1.4}, we state Theorem~\ref{thm-1.4} -- this
is the first result of the paper. It gives the associated Euler-Lagrange equations
for the Gauss-Bonnet integrand for
manifolds with boundary.

The remainder of the paper is devoted to the proof of Theorem~\ref{thm-1.4}.
Section~\ref{sect-2} treats basic invariance theory; the question of universal
curvature identities is central. In Theorem~\ref{thm-2.3}, we shall summarize
previous
results concerning universal curvature identities in the scalar case
(both in the interior and on the boundary) and in the symmetric 2-tensor
case in the interior. Theorem~\ref{thm-2.4}
is the second main result of this paper. In it, we extend the results of
Theorem~\ref{thm-2.3} to discuss universal curvature identities
for symmetric 2-tensors defined by the geometry of the embedding $\partial M\subset M$.
There is a technical fact we shall need
in the proof of Theorem~\ref{thm-2.4} that we postpone until Section~\ref{sect-4}
to avoid breaking the flow of the discussion. In Section~\ref{sect-3},
we use Theorem~\ref{thm-2.4} to complete the proof of Theorem~\ref{thm-1.4}.

Section~\ref{sect-4} provides the technical results which are
central to the discussion. Rather than using H. Weyl's theory of
invariants \cite{W46} to treat the universal curvature identities
which arise in Theorem~\ref{thm-1.4}, we have chosen to adopt the
approach of \cite{G73} which was originally used to give a
heat equation proof of the Gauss-Bonnet theorem. This seemed
easier rather than having to introduce an additional complicated
formalism to use results of \cite{W46}.

\subsection{The Gauss-Bonnet Theorem for closed manifolds}\label{sect-1.1}
Let $(M,g)$ be a compact pseudo-Riemannian manifold of signature
$(p,q)$ and dimension $m=p+q$ with smooth boundary $\partial M$; if the
signature is indefinite, we assume the restriction of the metric to
the boundary to be non-degenerate.
Let $dx_g$ be the Riemannian element of volume.
Let $\vec e:=\{e_1,...,e_m\}$ be a local orthonormal frame for
the tangent bundle of $M$. Set
\begin{equation}\label{eqn-1.a}
\begin{array}{l}
\varepsilon_{i_1...i_n}^{j_1...j_n}
=\varepsilon(g,\vec e\,)_{i_1...i_n}^{j_1...j_n}:=
g(e^{i_1}\wedge...\wedge e^{i_n},
e^{j_1}\wedge...\wedge e^{j_n})\\
\qquad\qquad\qquad\qquad\qquad
=\det\left(g(e^{i_\mu},e^{j_\nu})\right)=\pm1\,.\vphantom{\vrule height 11pt}
\end{array}\end{equation}
Clearly this vanishes if the indices are not distinct and thus, in particular, is zero
if $n>m$. We adopt the {\it Einstein convention} and
sum over repeated indices. Let $R_{ijkl}=R(g,\vec e\,)_{ijkl}$ denote components of
the curvature
tensor of the Levi-Civita connection $\nabla^g$ relative to the local
orthonormal frame $\vec e$. If $n=2\bar n$ is even, define
the {\it Euler form} or the {\it Pffaffian}
$E_{m,n}^{(p,q)}=E_{m,n}^{(p,q)}(g)$ by setting:
\begin{equation}\label{eqn-1.b}
E_{m,n}^{(p,q)}:=\frac
{{R_{i_1i_2j_2j_1}...R_{i_{n-1}i_{n}j_nj_{n-1}}}}
{{{{(8\pi)}}^{\bar n}\bar
n!}}\varepsilon_{j_1...j_n}^{i_1...i_n}\,.
\end{equation}
This is independent of the choice of the local orthonormal frame field $\vec e$.
We set $E_{m,n}^{(p,q)}=0$ if $n$ is odd. Let $\chi(M)$ be the
Euler-Poincar\'e characteristic. The generalized
Gauss-Bonnet theorem \cite{AW43,C44,C63} states:

\begin{theorem}\label{thm-1.1}
Let $(M,g)$ be a closed pseudo-Riemannian
manifold of signature $(p,q)$ and dimension $m=p+q$. Then
$$\chi(M)=\int_M E_{m,m}^{(p,q)}(g)dx_g\,.$$
\end{theorem}

We note that $\chi(M)=0$ in this setting if $m$ is odd and it was for that reason
we set $E_{m,n}^{(p,q)}=0$ if $n$ was odd. If $p$ or $q$ is odd, then
$\chi(M)=0$ even though $E_{m,m}^{(p,q)}(g)$ need not vanish locally
in this setting.

\subsection{The Euler-Lagrange Equations for manifolds without boundary}
\label{sect-1.2}
We examine these formulas in dimensions $m>n$. Let $h$ be a symmetric
$2$-cotensor. We consider the variation of the metric
$g_t:=g+t h$; this is a non-degenerate metric of signature $(p,q)$ for small $t$.
We integrate by parts to express:
$$
\displaystyle\partial_t\left.\left \{\int_ME_{m,n}^{(p,q)}(g_t)dx_{g_t}\right\}\right|_{t=0} =\int_M
h_{ij}Q_{m,n,ij}^{(p,q),2}(g)dx_g\,.
$$
Here $Q_{m,n,ij}^{(p,q),2}=Q_{m,n,ij}^{(p,q),2}(g)$ is a canonically
defined symmetric 2-tensor field. Since $E_{m,n}^{(p,q)}=0$ for
$m<n$ or for $n$ odd, we set $Q_{m,n,ij}^{(p,q),2}=0$  in these
cases. Furthermore, Theorem~\ref{thm-1.1} shows
$Q_{m,n,ij}^{(p,q),2}=0$ if $m=n$. Thus only the case $m>n$ and $n$
even is relevant.

In the Riemannian setting, Berger \cite{M70} conjectured that $Q_{m,n,ij}^{(p,q),2}$
could be expressed only in terms of curvature; the higher covariant
derivatives did not enter. This was subsequently verified by
Kuz'mina \cite{K74} and Labbi \cite{L05,L07,L08}. Set
$\mathcal{E}_{m,n}^{(p,q),2}=0$ if $n$ is odd. If $n=2\bar n$ is
even, define
$\mathcal{E}_{m,n,ij}^{(p,q),2}=\mathcal{E}_{m,n,ij}^{(p,q),2}(g)$
by setting:
\begin{equation}\label{eqn-1.c}
\mathcal{E}_{m,n,ij}^{(p,q),2}:=\frac
{R_{i_1i_2j_2j_1}...R_{i_{n-1}i_{n}j_{n}j_{n-1}}}
{{{(8\pi)}^{\bar n}\bar
n!}}\varepsilon_{jj_1...j_n}^{ii_1...i_n}\,.
\end{equation}
This symmetric 2-tensor valued function is
 independent of the choice of $\vec e$. One then has \cite{GPS11,GPS12}:

\begin{theorem}\label{thm-1.2}
If $M$ is a closed pseudo-Riemannian manifold
of signature $(p,q)$ and dimension $m=p+q>n$, then:
$$\partial_t\left.\left\{\int_ME_{m,n}^{(p,q)}(g_t)
dx_{g_t}\right\}\right|_{t=0}
=\frac12\int_Mh_{ij}\mathcal{E}_{m,n,ij}^{(p,q),2}(g)dx_g.
$$
\end{theorem}

\subsection{The Gauss-Bonnet Theorem for manifolds with boundary}
\label{sect-1.3}
There are boundary correction terms which appear if $\partial M$ is non-empty.
We assume the restriction of the pseudo-Riemannian metric $g$ to the boundary
is non-degenerate. Normalize the choice of $\vec e$ near $\partial M$
so that $e_1$ is the inward pointing unit geodesic
vector field. Let indices $\{a,b\}$ range from $2$ through $m$ and index the
induced orthonormal frame for the tangent bundle of the boundary.
Let $\nabla$ denote the Levi-Civita connection.
Let $L_{ab}:=L(g,\vec e\,)_{ab}$ give the components of the {\it second fundamental form}:
$$L_{ab}:=g(\nabla_{e_a}e_b,e_1)\,.$$
The parity of $n$ plays no role.
For $0\le2\nu\le n-1$, set:
\begin{eqnarray}
F_{m,n-1,\nu}^{(p,q),\partial
M}&:=&\frac{R_{a_1a_2{{b_2b_1}}}...
    {{R_{a_{2\nu-1}a_{2\nu}b_{2\nu}b_{2\nu-1}}}}
    L_{a_{2\nu+1}b_{2\nu+1}}...L_{a_{n-1}b_{n-1}}}
    {{{(8\pi)}}^\nu \nu!\operatorname{Vol}
(S^{n-1-2\nu})(n-1-2\nu)!}\varepsilon_{a_1...a_{n-1}}^{b_1...b_{n-1}},
\nonumber\\
F_{m,n-1}^{(p,q),\partial M}
&:=&\sum_{\nu}F_{m,n-1,\nu}^{(p,q),\partial M}\,.\label{eqn-1.d}
\end{eqnarray}
The scalar functions $\{F_{m,n-1}^{(p,q),\partial M},F_{m,n-1,\nu}^{(p,q),\partial M}\}$ are independent of the choice of the local
orthonormal frame $\vec e$.
Theorem~\ref{thm-1.1} generalizes to this setting to yield:

\begin{theorem} Let $(M,g)$ be a compact pseudo-Riemannian
manifold of signature $(p,q)$ and of dimension $m=p+q$.
Assume the restriction of the metric to the
boundary is non-degenerate. Then
$$\chi(M)=\int_M E_{m,m}^{(p,q)}(g)dx_g
+\int_{\partial M}F_{m,m-1}^{(p,q),\partial M}(g)dy_g\,.$$
\end{theorem}

\subsection{The Euler-Lagrange Equations for manifolds with boundary}
\label{sect-1.4}
Assume the restriction of the pseudo-Riemannian metric to the boundary
is non-degenerate.
For $0\le2\nu\le n-1$, define:
\begin{eqnarray}
\mathcal{F}_{m,n-1,\nu,ab}^{(p,q),2,\partial M}&:=&
\frac{R_{a_1a_2{{b_2b_1}}}...
    {{R_{a_{2\nu-1}a_{2\nu}{b_{2\nu}}b_{2\nu-1}}}}
    L_{a_{2\nu+1}b_{2\nu+1}}..L_{a_{n-1}b_{n-1}}}
    {({{8\pi}})^\nu \nu!\operatorname{Vol}
(S^{n-1-2\nu})(n-1-2\nu)!}\varepsilon_{bb_1...b_{n-1}}^{aa_1...a_{n-1}},
\nonumber\\
\mathcal{F}_{m,{n-1},ab}^{(p,q),2,\partial M}
     &:=&\sum_\nu\mathcal{F}_{m,{n-1},{{\nu}},ab}^{(p,q),2,\partial M}\,.
\label{eqn-1.e}
\end{eqnarray}
These symmetric 2-tensor valued functions on the boundary
are independent of the choice
of the local orthonormal frame $\vec e$.
The first main new result of this paper is to generalize Theorem~\ref{thm-1.2}
to the case of manifolds with boundary; the remainder of this paper is
devoted to the proof of the following result:

\begin{theorem}\label{thm-1.4}
Let $M$ be a compact pseudo-Riemannian manifold with
boundary of dimension $m\ge n+1$. Assume the restriction of the
metric to the boundary is non-degenerate. Then:
\begin{eqnarray*}
&&\partial_t\left.\left\{\int_M
E_{m,n}^{(p,q)}(g_t)dx_{g_t}
+\int_{\partial M}F_{m,n-1}^{(p,q),\partial M}(g_t)
dy_{g_t}\right\}\right |_{t=0}\\
&=& \frac12\int_Mh_{ij}\mathcal{E}_{m,n,ij}^{(p,q),2}(g)dx_g+
\frac12\int_{\partial
M}h_{ab}\mathcal{F}_{m,{n-1},ab}^{(p,q),2,\partial
M}(g)dy_g\,.
\end{eqnarray*}
\end{theorem}

\section{Invariance theory}\label{sect-2}
Section~\ref{sect-2} is devoted to the discussion of invariance theory.
We begin in Section~\ref{sect-2.1} with a discussion of local formulas;
this is central to the matter at hand.
In Section~\ref{sect-2.2}, we introduce the spaces of invariants
with which we shall be working. We shall be dealing with formal expressions
in the covariant derivatives of the curvature tensor (and also the tangential
covariant derivatives of the second fundamental form when considering
the boundary geometry) which are invariant under
the action of the orthogonal group (and thus independent of the choice
of local orthonormal frame $\vec e$) and which are either scalar valued
or symmetric 2-tensor valued. They are universally defined by contractions
of indices. In Section~\ref{sect-2.3}, we give various examples of
scalar and symmetric 2-tensor valued invariants both in the interior and
on the boundary. We also begin the discussion of universal curvature identities.
In Section~\ref{sect-2.4}, we
introduce the restriction map. An expression which is non-zero in dimension $m$
may vanish when restricted to dimension $m-1$ and thus becomes a
universal curvature identity in dimension $m-1$. We give both a geometric
definition and then subsequently an algebraic definition of the restriction map.
We analyze in
Theorem~\ref{thm-2.3} the kernel of the restriction map in certain contexts.
Section~\ref{sect-2.5} treats symmetric 2-tensor valued invariants on the
boundary and contains, in Theorem~\ref{thm-2.4}, the second main result
of this paper which is of independent interest.

\subsection{Local formulas}\label{sect-2.1} It is worth saying a bit about
the general framework in which we shall be working. We first consider invariants
defined on the interior of a pseudo-Riemannian manifold $(M,g)$.
Let $X=(x_1,...,x_m)$ be a system of local coordinates on
$M$.
Let $\alpha$ be a multi-index. We introduce formal variables
$$g_{ij}:=g(\partial_{x_i},\partial_{x_j})\quad\text{and}\quad
g_{ij/\alpha}:=d_x^\alpha g_{ij}$$
for the derivatives of the metric.
We consider expressions $P=P(g_{ij},g_{ij/\alpha})$
which are polynomial in the variables $\{g_{ij/\alpha}\}_{|\alpha|>0}$ with
coefficients which depend smoothly on the $\{g_{ij}\}$ variables. Given
a system of local coordinates $X$ and a point $P$, we can evaluate
$P(g,X)$; we say $P$ is {\it invariant} if this evaluation is independent of
the particular coordinate system chosen and depends only on the point of the
manifold and on the metric $g$. We introduce special notation for the first and second derivatives
of the metric setting:
$$g_{ij/k}:=\partial_{x_k}g_{ij}\quad\text{and}\quad
   g_{ij/kl}:=\partial_{x_k}\partial_{x_l}g_{ij}\,.$$
If we take geodesic polar coordinates
centered at a point $x_0$ of $M$, then we have
$$g_{ij}(g,X)(x_0)=\pm\delta_{ij}\quad\text{and}\quad
   g_{ij/k}(g,X)(x_0)=0\,.$$
Furthermore, all the higher derivatives of the metric at $x_0$ can
be evaluated in terms of the components of the covariant derivatives of
the curvature tensor at $x_0$. Thus we can equally well regard
$P=P(g_{ij},R_{ijkl},R_{ijkl;n},...)$. This reduces the structure group
from the group of germs of diffeomorphisms to the orthogonal group;
the first theorem of orthogonal invariants of H. Weyl \cite{W46} then can
be used to show all invariants arise from contracting indices. Conversely, of course,
we can recover the original formalism by expressing the curvature and
its covariant derivatives in terms of the derivatives of the metric. These
are two equivalent points of view and both will play a role in what follows.
But what is important is that we are thinking of $P$ as a local formula which
can be evaluated on a metric and at a point.

\subsection{Spaces of invariants}\label{sect-2.2}
We say that the curvature $R_{ijkl}$ is of degree 2 since it is
linear in the 2-jets of the metric and quadratic in the $1$-jets of the metric.
The second fundamental form is of degree 1. Each covariant
derivative increases the degree by 1.
We define the spaces with which we shall be working as follows:
\begin{definition}
\rm\ \begin{enumerate}
\item   Let $\mathcal{I}_{m,n}^{(p,q)}$ be the space of
 invariants of degree $n$ that are scalar
 valued and that are defined in the category of all pseudo-Riemannian manifolds of
 signature $(p,q)$ and of  dimension $m=p+q$.
\item Let $\mathcal{I}_{m,n}^{(p,q),2}$ be the space of
 invariants of degree $n$ which are symmetric
$2$-tensor valued and that are defined in the category
 of all pseudo-Riemannian manifolds of
 signature $(p,q)$ and  of dimension $m=p+q$.
 \item   Let $\mathcal{I}_{m,n}^{(p,q),\partial M}$ be the space of
 invariants of degree $n$ that are scalar
 valued and that are defined on the boundary
which are defined in the category of all pseudo-Riemannian manifolds of
 signature $(p,q)$ and of dimension $m=p+q$.
  \item   Let $\mathcal{I}_{m,n}^{(p,q),2,\partial M}$ be the space of
 invariants of degree $n$ that are
 symmetric $2$-tensor valued and defined on the boundary and
that are defined in the category of all pseudo-Riemannian manifolds of
 signature $(p,q)$ and of dimension $m=p+q$.
\end{enumerate}\end{definition}

By H. Weyl's first theorem of invariants, all invariants arise from contraction of indices. For example, we have (see, for example, \cite{G75x}):

\begin{lemma}\label{lem-2.1}
If the metric is positive definite, then:\begin{enumerate}
\item
$\mathcal{I}_{m,0}^{(0,m)}=\operatorname{Span}\left\{1\right\}$.
\smallbreak\item
$\mathcal{I}_{m,2}^{(0,m)}=\operatorname{Span}\left\{R_{ijji}\right\}$.
\smallbreak\item
$\mathcal{I}_{m,4}^{(0,m)}=\operatorname{Span}\left\{R_{ijji;kk},R_{ijji}R_{kllk},
R_{ijjk}R_{illk}, R_{ijkl}R_{ijkl}\right\}$.
\smallbreak\item
$\mathcal{I}_{m,6}^{(0,m)}=\operatorname{Span}\left\{R_{ijji;kkll},R_{ijji;k}R_{lnnl;k},R_{aija;k}R_{bijb;k}\right.$,
$R_{ajka;n}R_{bjnb;k},$\smallbreak
$R_{ijkl;n}R_{ijkl;n},R_{ijji}R_{kllk;nn},R_{ajka}R_{bjkb;nn},R_{ajka}R_{bjnb;kn}$,
\smallbreak
$R_{ijkl}R_{ijkl;nn},R_{ijji}R_{kllk}R_{abba},R_{ijji}R_{ajka}R_{bjkb},
R_{ijji}R_{abcd}R_{abcd}$, \smallbreak
$R_{ajka}R_{bjnb}R_{cknc},R_{aija}R_{bklb}R_{ikjl},
R_{ajka}R_{jnli}R_{knli},R_{ijkn}R_{ijlp}R_{knlp}$, \smallbreak
$\left.R_{ijkn}R_{ilkp}R_{jlnp}\right\}$.
\end{enumerate}\end{lemma}

\begin{remark}
\rm In the indefinite setting, we must take more care with raising
and lowering indices
and there is a slight additional bit of technical fuss. If $\{e_i\}$ is an orthonormal
 basis, let $\xi_i:=g(e_i,e_i)=\pm1$. We must
then, for example, replace Assertions (1)-(3) of Lemma~\ref{lem-2.1} by:
\begin{eqnarray*}
&&\mathcal{I}_{m,0}^{(p,q)}
=\operatorname{Span}\{1\},\quad
    \mathcal{I}_{m,2}^{(p,q)}
    =\operatorname{Span}\{\textstyle\sum_{ij}\xi_i\xi_jR_{ijji}\},\\
&&\mathcal{I}_{m,4}^{(p,q)}
=\operatorname{Span}
\{\textstyle\sum_{ijk}\xi_i\xi_j\xi_kR_{ijji;kk},
\textstyle\sum_{ijkl}\xi_i\xi_j\xi_k\xi_lR_{ijji}R_{kllk},\\
&&\qquad\qquad
\textstyle\sum_{ijkl}\xi_i\xi_j\xi_k\xi_lR_{ijjk}R_{illk},
\textstyle\sum_{ijkl}\xi_i\xi_j\xi_k\xi_lR_{ijkl}R_{ijkl}\}\,.
\end{eqnarray*}\end{remark}

If $\eta_1$ and $\eta_2$ are covectors, define the symmetric product
$\eta_1\circ\eta_2$ by setting:
$$\eta_1\circ\eta_2=\textstyle\frac{1}{2}(\eta_1\otimes \eta_2+\eta_2\otimes\eta_1)\,.$$
The metric takes the form $g=e^i\circ e^i$; the map
$P\rightarrow P\cdot g$ embeds $\mathcal{I}_{m,n}^{(p,q)}$ into
$\mathcal{I}_{m,n}^{(p,q),2}$. Lemma~\ref{lem-2.1} generalizes to be \cite{GPS11}:
\begin{lemma}\label{lem-2.2}
If the metric is positive definite, then:
\begin{enumerate}
\item $\mathcal{I}_{m,0}^{(0,m),2}=\operatorname{Span}\left\{e^k\circ e^k\right\}$.
\smallbreak\item
$\mathcal{I}_{m,2}^{{(0,m),2}}=\operatorname{Span}\left\{R_{ijji}e^k\circ
e^k,R_{ijki}e^j\circ e^k\right\}$. \smallbreak\item
$\mathcal{I}_{m,4}^{{(0,m),2}}=\operatorname{Span}\left\{R_{ijji;kk}e^l\circ
e^l\right.$, $R_{kjjl;ii}e^k\circ e^l$, $R_{ijji;kl}e^k\circ e^l$,
\smallbreak\quad
$R_{ijji}R_{kllk}e^n\circ e^n$,
$R_{ijki}R_{ljkl}e^n\circ e^n$,
$R_{ijkl}R_{ijkl}e^n\circ e^n$, $R_{ijji}R_{klnk}e^l\circ e^n$,
\smallbreak\quad
$R_{ikli}R_{jknj}e^l\circ e^n$,
$R_{ijkl}R_{ijkn}e^l\circ e^n$, $\left.R_{lijn}R_{kijk}e^l\circ e^n\right\}$.
\end{enumerate}
\end{lemma}

Near the boundary, we normalized the local orthonormal frame so that $e_1$
is the inward unit geodesic normal vector field. Let indices $\{a,b\}$ range from $2$
through $m$ and index the induced orthonormal frame $\{e_2,...,e_m\}$ for the
tangent bundle of the boundary. Let ``$:$" denote multiple covariant differentiation with respect to the Levi-Civita
connection of the boundary. We have \cite{BG90}:
\begin{lemma}\label{lem-2.3}
If the metric is positive definite, then:\begin{enumerate}
\item $\mathcal{I}_{m,0}^{(0,m),\partial M}=
\operatorname{Span}\{1\}$.
\item $\mathcal{I}_{m,1}^{(0,m),\partial M}=
\operatorname{Span}\{L_{aa}\}$.
\item $\mathcal{I}_{m,2}^{(0,m),\partial M}=
\operatorname{Span}\{R_{abba},R_{a11a},
L_{aa}L_{bb},L_{ab}L_{ab}\}$.
\item $\mathcal{I}_{m,3}^{(0,m),\partial M}=
\operatorname{Span}\{R_{a11a;1}, L_{aa:bb}, L_{ab:ab},
R_{bccb}L_{aa}, R_{a11a}L_{bb}, R_{a11b}L_{ab},$
\smallbreak\quad $R_{abcb}L_{ac}, L_{aa}L_{bb}L_{cc},
L_{ab}L_{ab}L_{cc}, L_{ab}L_{bc}L_{ac}\}$.
\end{enumerate}
\end{lemma}
 It is not difficult to establish the following result to
round out our treatment; we omit the proof in the interests of
brevity since we shall not need it in discussion:

%%%
\begin{lemma}\label{lem-2.4}
If the metric is positive definite, then:
\begin{enumerate}
\item $\mathcal{I}_{m,0}^{(0,m),2,\partial M}
=\operatorname{Span}\{e^1\circ e^1,e^a\circ e^a\}$.
\item $\mathcal{I}_{m,1}^{(0,m),2,\partial M}
=\operatorname{Span}\{L_{aa}e^1\circ e^1,L_{aa}e^b\circ e^b,
L_{ab}e^a\circ e^b\}$.
\item $\mathcal{I}_{m,2}^{(0,m),2,\partial M}
=\operatorname{Span}\{R_{a11a}e^1\circ e^1,L_{aa}L_{bb}e^1\circ e^1,
L_{ab}L_{ab}e^1\circ e^1$, \smallbreak\quad $R_{a1ba}e^1\circ
e^b,L_{ab:a}e^1\circ e^b, L_{aa:b}e^1\circ e^b$, $R_{abba}e^c\circ
e^c,R_{a11a}e^b\circ e^b,$\smallbreak\quad $ L_{aa}L_{bb}e^c\circ
e^c, L_{ab}L_{ab}e^c\circ e^c$, $R_{abca}e^b\circ e^c,
R_{a11b}e^a\circ e^b,L_{ab}L_{ac}e^b\circ e^c,$,\smallbreak\quad
$L_{aa}L_{bc}e^b\circ e^c\}$.
\item $\mathcal{I}^{(0,m),3,\partial M}_{m,3}
=\operatorname{Span}\{R_{a1a1;1}e^1\circ e^1,R_{a1a1;1}e^b\circ
e^b,R_{b1c1;1}e^b\circ e^c$,\smallbreak\quad $R_{abbc;a}e^c\circ
e^1,R_{a11a;c}e^c\circ e^1,R_{a11c;a}e^c\circ e^1, L_{aa : bb}
e^{c}\circ e^{c}$, \smallbreak\quad $L_{aa : bb} e^{1}\circ
e^{1},L_{ab : ab} e^{c}\circ e^{c}, L_{ab : ab} e^{1}\circ
e^{1},L_{aa : bc} e^{b}\circ e^{c},L_{ab : ac} e^{b}\circ e^{c},$
\smallbreak\quad $L_{bc:aa} e^{b}\circ e^{c},
R_{abab}L_{cc} e^d\circ e^d, R_{abac}L_{bc} e^d\circ e^d,
R_{a1a1}L_{bb}e^c\circ e^c$, \\
\indent\indent $R_{abab}L_{cc} e^1 \circ e^1,R_{abac}L_{bc}e^1\circ
e^1,R_{a1a1}L_{bb}e^1\circ e^1$,\smallbreak\quad$
R_{a1b1}L_{ab}e^c\circ e^c,R_{a1b1}L_{ab}e^1\circ e^1,R_{abab}L_{cd}
e^c \circ e^d,
 R_{a1a1}L_{cd}e^c\circ {{e^d}},$\smallbreak\quad$R_{acbd}L_{ab}e^c\circ e^d,R_{dbcb}L_{aa}e^c\circ e^d,
{ R_{c1d1}L_{aa} e^c\circ e^d}, R_{c1a1}L_{ad}e^c\circ
e^d,$\smallbreak\quad $R_{cbab}L_{ad}e^c\circ e^d,L_{aa: c}L_{bc}
e^{b}\circ e^{1},L_{aa: c}L_{bb} e^{c}\circ e^{1},L_{ac: b}L_{ab}
e^{c}\circ e^{1},$ \smallbreak\quad $L_{ac:a}L_{bc}e^b\circ e^1,
L_{ac:a}L_{bb}e^c\circ e^1, L_{ab:c}L_{ab}e^c\circ e^1,
L_{aa}L_{bb}L_{cc} e^{d}\circ e^{d} $\smallbreak\quad $
L_{ab}L_{ab}L_{cc} e^{d}\circ e^{d},L_{ab}L_{bc}L_{ca}e^{d}\circ
e^{d}, L_{aa}L_{bb}L_{cc} e^{1}\circ e^{1},$
\smallbreak\quad$L_{ab}L_{ab}L_{cc} e^{1}\circ
e^{1},L_{ab}L_{bc}L_{ca} e^{1}\circ e^{1}, L_{aa}L_{bb}L_{cd}
e^{c}\circ e^{d},$\smallbreak\quad$L_{aa}L_{bc}L_{bd}e^c\circ d^d,
L_{ab}L_{ab}L_{cd} e^{c}\circ e^{d}, L_{bc}L_{ab}L_{ad} e^{c}\circ
e^{d}\}.$
\end{enumerate}
\end{lemma}

%%%

\subsection{Universal curvature identities}\label{sect-2.3}
The invariants of Lemma~\ref{lem-2.1} are not linearly independent in low
dimensions. To simplify matters, we work in the Riemannian context; similar
results hold if we introduce the tensor $\xi$. We have the following
relations in dimensions $1$, $3$, and $5$:
\begin{lemma}\label{lem-2.5}
In the Riemannian setting: \begin{enumerate}
\item If $m=1$, then $0=R_{ijji}$.
\smallbreak\item If $m=3$, then $0=R_{ijji}R_{kllk}-4R_{aija}R_{bijb}+R_{ijkl}R_{ijkl}$.
\smallbreak\item If $m=5$, then $0=R_{ijji}R_{kllk}R_{abba}-12R_{ijji}R_{aija}R_{bijb}+3R_{abba}R_{ijkl}R_{ijkl}$
\smallbreak\qquad\quad
$+24R_{aija}R_{bklb}R_{jlik}+16R_{aija}R_{bjkb}R_{cikc}-24R_{aija}R_{jkln}R_{lnik}$
\smallbreak\qquad\quad
$+2R_{ijkl}R_{klan}R_{anij}-8R_{kaij}R_{inkl}R_{jlan}$.
\end{enumerate}\end{lemma}

Let $n$ be even. If $\sigma$ is a permutation and if $R$ is the
curvature tensor, let
$$P_{n,\sigma}(R):=g^{i_1i_2}\cdot\cdot\cdot g^{i_{2n-1}i_{2n}}
R_{i_{\sigma_1}i_{\sigma_2}i_{\sigma_3}i_{\sigma_4}}\cdot\cdot\cdot
R_{i_{\sigma_{2n-3}}i_{\sigma_{2n-2}}i_{\sigma_{2n-1}}i_{\sigma_{2n}}}
\,.$$
If $\mathcal{C}=\{C_\sigma\}$ is a collection of constants,
let:
$$
P_{n,\mathcal{C}}:
=\sum_{\sigma\in\operatorname{Perm}(4\ell)} c_\sigma
P_{n,\sigma}\,.
$$
Let $(p,q)$ be arbitrary.
H. Weyl's first theorem of invariants, when applied to the curvature tensor, yields
that if $P\in\mathcal{I}_{m,n}^{(p,q)}$ is a scalar
invariant of the curvature tensor of degree $n$ which does not involve
the covariant derivatives of the curvature tensor,
then $P=P_{n,\mathcal{C}}$ for some
$\mathcal{C}$; if $P\ne0$, necessarily $n$ is even. Similar expressions appear
 if the covariant derivatives of $R$
are involved; the notation becomes more complicated and as we shall not
in any event need such expressions, we omit details in the interests of brevity.

We say that $P_{n,\mathcal{C}}$ is a {\it universal curvature
identity in signature $(p,q)$} if $P_{n,\mathcal{C}}$ vanishes for
all pseudo-Riemannian metrics of signature $(p,q)$. The following is
a useful observation  \cite{GPS12}; it shows only the
dimension is relevant.
\begin{theorem}\label{thm-2.2}
If $P_{n,\mathcal{C}}$ is a universal scalar curvature identity in the curvature tensor
 in signature $(p,q)$,
then $P_{n,\mathcal{C}}$ is a universal scalar curvature
in any other signature $(\tilde p,\tilde q)$ if $p+q=\tilde p+\tilde q$.
\end{theorem}

\begin{remark} \rm
A similar assertion holds for elements of
$\mathcal{I}_{m,n}^{(p,q),2}$. If we consider polynomials in
$\{L,R\}$, a similar assertion holds for elements of
$\mathcal{I}_{m,n}^{(p,q),\partial M}$ and
$\mathcal{I}_{m,n}^{(p,q),2,\partial M}$. In the discussion of
\cite{GPS12}, we first passed to the algebraic setting and then used
analytic continuation. The reason to avoid dealing with the
covariant derivatives of the curvature tensor was the relation
\begin{eqnarray*}
&&R_{ijkl;uv}-R_{ijkl;vu}\\
&=&R_{uviw}R_{wjkl}+R_{uvjw}R_{iwkl}+R_{uvkw}R_{ijwl}+R_{uvlw}R_{ijkw}\,.
\end{eqnarray*}
This is quadratic in the curvature.
Thus the space of possible tensors is not a linear space.
As we shall not need Theorem~\ref{thm-2.2}, we shall omit the proof and present it
simply for the sake of completeness.
\end{remark}

There are {\it trivial} identities that arise from the curvature symmetries
\begin{equation}\label{eqn-2.a}
R_{ijkl}=-R_{jikl}=R_{klij}\quad\text{and}\quad
    R_{ijkl}+R_{jkil}+R_{kijl}=0\,.
\end{equation}
Thus, for example, $R_{ijji}+R_{ijij}=0$ is a universal curvature identity;
this expression defines the zero local formula. The
identities of Lemma~\ref{lem-2.5} do not arise in this fashion; they are dimension
specific. These local formulas are zero in dimensions $\{1,3,5\}$ but are non-zero in
dimensions $\{2,4,6\}$, respectively.

\subsection{The restriction map}\label{sect-2.4}
We introduce some additional notation to describe universal curvature identities
which are dimension specific. Fix a signature $(p,q)$.
Let $P\in\mathcal{I}_{m,n}^{(p,q)}$. Set
$s_-(p,q)=(p-1,q)$ and $s_+(p,q)=(p,q-1)$. If $p=0$, then set $r_-(P)=0$ and if
$q=0$, then set $r_+(P)=0$ since there are no manifolds of this signature.
Otherwise, let
$(N_\pm^{m-1},g_N)$ be an $m-1$ dimensional pseudo-Riemannian manifold
of signature $s_\pm(p,q)$. Define $r_\pm(P)\in\mathcal{I}_{m-1,n}^{s_\pm(p,q)}$
by setting:
$$r_\pm(P)(N^{m-1},g_N)(x)=P(N^{m-1}\times S^1,g_N\pm d\theta^2)(x,\theta)\,.$$
The particular angle $\theta\in S^1$ is irrelevant as $S^1$ is a homogeneous
space. This yields an invariant in dimension $m-1$ of the appropriate signature and
defines maps:
$$r_-:\mathcal{I}_{m,n}^{(p,q)}\rightarrow\mathcal{I}_{m-1,n}^{(p-1,q)}
\quad\text{and}\quad
r_+:\mathcal{I}_{m,n}^{(p,q)}\rightarrow\mathcal{I}_{m-1,n}^{(p,q-1)}\,.
$$
We use a similar construction to define
$$\begin{array}{ll}
r_-:\mathcal{I}_{m,n}^{(p,q),2}\rightarrow
          \mathcal{I}_{m-1,n}^{(p-1,q),2},&
r_+:\mathcal{I}_{m,n}^{(p,q),2}\rightarrow
             \mathcal{I}_{m-1,n}^{(p,q-1),2},\\
r_-:\mathcal{I}_{m,n}^{(p,q),\partial M}\rightarrow
           \mathcal{I}_{m-1,n}^{(p-1,q),\partial M},&
r_+:\mathcal{I}_{m,n}^{(p,q),\partial M}\rightarrow
             \mathcal{I}_{m-1,n}^{(p,q-1),\partial M},\\
\vphantom{\vrule height 11pt}
r_-:\mathcal{I}_{m,n}^{(p,q),2,\partial M}\rightarrow
          \mathcal{I}_{m-1,n}^{(p-1,q),2,\partial M},&
r_+:\mathcal{I}_{m,n}^{(p,q),2,\partial M}\rightarrow
             \mathcal{I}_{m-1,n}^{(p,q-1),2,\partial M}.
\vphantom{\vrule height 11pt}\end{array}$$
We emphasize that if $p=0$, then $r_-=0$ and if $q=0$, then $r_+=0$ since
there are no manifolds of the indicated signature.

 Elements of $\mathcal{I}_{m,n}^{(p,q)}$ are defined by summations that range
 from $1$ to $m$; the corresponding elements of
 $\mathcal{I}_{m-1,n}^{(p,q)}$
 are defined by restricting the range of the summation. For example:
 \smallbreak\centerline{
$\displaystyle\tau_m=\sum_{i,j=1}^mR_{ijji}$ and
$\displaystyle\tau_{m-1}=\sum_{i,j=1}^{m-1}R_{ijji}$.}
\smallbreak\noindent Thus $r(\tau_m)=\tau_{m-1}$ since the form of
the invariant is the same. The scalar curvature is universal in this
sense - it is invariant under restriction. Furthermore, it is now
clear that the restriction of a universal curvature identity in
dimension $m$ generates a corresponding universal curvature identity
in dimension $m-1$. One has \cite{G73,G75,GPS12}:

\begin{theorem}\label{thm-2.3}
Adopt the notation established above.
\begin{enumerate}
\item If $n<m$, then
$\ker\{r_+\}\cap\ker\{r_-\}\cap\mathcal{I}_{m,n}^{(p,q)}=\{0\}$.
\smallbreak\item If $m$ is odd, then $\mathcal{I}_{m,m}^{(p,q)}=\{0\}$.
If $m$ is even, then the invariant $\{E_{m,m}^{(p,q)}\}$
of Equation~(\ref{eqn-1.b}) is a basis for
 $\ker\{r_+\}\cap\ker\{r_-\}\cap\mathcal{I}_{m,m}^{(p,q)}$.
\smallbreak\item If $n<m-1$, then
$\ker\{r_+\}\cap\ker\{r_-\}\cap\mathcal{I}_{m,n}^{(p,q),2}=\{0\}$.
\smallbreak\item If $m$ is even, then
$\mathcal{I}_{m,m-1}^{(p,q),2}=\{0\}$. If $m$ is odd, then the
invariant $\{\mathcal{E}_{m,m-1}^{(p,q),2}\}$ of
Equation~(\ref{eqn-1.c}) is a basis for
$\ker\{r_+\}\cap\ker\{r_-\}\cap\mathcal{I}_{m,m-1}^{(p,q),2}$.
\smallbreak\item If $n<m-1$, then
$\ker\{r_+\}\cap\ker\{r_-\}\cap\mathcal{I}_{m,n}^{(p,q),\partial
M}=\{0\}$.
\smallbreak\item The invariants
$\{F_{m,m-1,\nu}^{(p,q),\partial M}\}$ of Equation~(\ref{eqn-1.d})
for $0\le 2\nu\le m-1$ are a basis for
$\ker\{r_+\}\cap\ker\{r_-\}\cap\mathcal{I}_{m,m-1}^{(p,q),\partial
M}$.
\end{enumerate}
\end{theorem}

\subsection{Symmetric 2-tensor valued invariants on the boundary}\label{sect-2.5}
The following result provides a characterization of the invariants
$\{\mathcal{F}_{m,n,\nu}^{(p,q),2,\partial M}\}$
which appear in Theorem~\ref{thm-1.4}.
It is the appropriate extension of Theorem~\ref{thm-2.3}
to this setting and is the second main result of this paper:

\begin{theorem}\label{thm-2.4}
Adopt the notation established above.
\begin{enumerate}
\item If $n<m-2$, then
$\ker\{r_+\}\cap\ker\{r_-\}\cap\mathcal{I}_{m,n}^{(p,q),2,\partial
M}=\{0\}$. \smallbreak\item The invariants
$\{\mathcal{F}_{m,m-2,\nu}^{(p,q),2,\partial M}\}$ of
Equation~(\ref{eqn-1.e}) for $0\le 2\nu\le m-2$ are a basis
for
$\ker\{r_+\}\cap\ker\{r_-\}\cap\mathcal{I}_{m,m-2}^{(p,q),2,\partial
M}$.
\end{enumerate}\end{theorem}

\begin{proof} Let $0\ne P\in\mathcal{I}_{m,n}^{(p,q),2,\partial M}$,
i.e. $P$ defines a non-zero local formula of degree $n$ on the
boundary in signature $(p,q)$. Express $P$ as a polynomial in the
derivatives of the metric. Fix a point $Y\in\partial M$. Recall that $e_1$ is the inward pointing geodesic normal vector
field. Choose
geodesic polar coordinates $(x_2,...,x_m)$ on the boundary centered
at $Y$ and then use the exponential map
$$\vec x\rightarrow
\exp_{(x_2,...,x_m)}(x_1e_1(x_2,...,x_m))$$
to define coordinates $(x_1,...,x_m)$ near $Y$.
Let $e_i:=\partial_{x_i}(Y)$; we choose the coordinates on the
boundary so that these form an orthonormal basis for $T_YM$ for
$1\le i\le m$ and the $e_i$ for $2\le i\le m$ form an orthonormal
basis for $T_Y(\partial M)$. We normalize the orthonormal basis so
that the indices $\{e_2,...,e_{\tilde q}\}$  are timelike and the
indices $\{e_{\tilde q+1},...,e_m\}$ are spacelike;  here $(\tilde
p,\tilde q)=(p-1,q)$ if the normal vector $e_1$ is timelike and
$(\tilde p,\tilde q)=(p,q-1)$ if the normal vector $e_1$ is spacelike. We
then have the relations:
$$
g_{1a}\equiv 0,\quad
g_{ab}(Y)=\pm\delta_{ab},\quad g_{ab/c}(Y)=0,\quad g_{ab/1}(Y)=L_{ab}\,.
$$

Let $P=P_{ij}e^i\circ e^j\in\mathcal{I}_{m,n}^{(p,q),2,\partial M}$.
Express $P_{ij}$ as a polynomial
in terms of the derivatives of the metric.
Let $A$ be a monomial of $P_{ij}$. After taking into account the normalizations
given above, we may express
\begin{eqnarray*}
&&A=g_{a_1b_1/1}...g_{a_kb_k/1}g_{i_1j_1/\alpha_1}...g_{i_\ell
j_\ell/\alpha_\ell} e^i\circ e^j
\quad\text{for}\\
&&n=k+|\alpha_1|+...+|\alpha_\ell|\quad\text{and}\quad |\alpha_i|\ge 2\text{ for }
1\le i\le\ell\,.
\end{eqnarray*}
If we assume that $r_+(P)=0$ and $r_-(P)=0$, then every
 index $a$ for $2\le a\le m$ must appear in $A$. In particular $\operatorname{deg}_a(A)\ge2$ for $2\le a\le m$. We count indices:
\begin{eqnarray*}
&&2(m-1)\le\sum_{a=2}^m\operatorname{deg}_a \{
g_{a_1b_1/1}...g_{a_kb_k/1}g_{i_1j_1/\alpha_1}...g_{i_\ell
j_\ell/\alpha_\ell}
e^i\circ e^j\}\\
&=&2+2k+\sum_{\nu=1}^\ell\sum_{a=2}^m
        \operatorname{deg}_a(g_{i_\nu j_\nu/\alpha_\nu})\le
        2+2k+\sum_{\nu=1}^\ell\sum_{i=1}^m
        \operatorname{deg}_i(g_{i_\nu j_\nu/\alpha_\nu})\\
&=&2+2k+2\ell+|\alpha_1|+...|\alpha_\ell|
\le2+ 2k+2(|\alpha_1|+...+|\alpha_\ell|)=2+2n\,.
\end{eqnarray*}
Consequently $m-1\le n+1$. This is, of course, not possible if $n<m-2$
which proves Assertion (1).

In the limiting case that $n=m-2$, all the inequalities given above must have been equalities. This implies that $\operatorname{deg}_1(A)=0$ and that $|\alpha_i|=2$
for $1\le i\le\ell$. Consequently,
\begin{equation}\label{eqn-2.b}
P=P_{ab}(g_{c_1c_2/1},g_{c_1c_2/c_3c_4})e^a\circ e^b\,.
\end{equation}
At the point in question, we can express the variables $g_{c_1c_2/c_3c_4}$ in
terms of the curvature $R^{\partial M}$.  Since $R^{\partial M}_{abcd}=R^M_{abcd}$
modulo quadratic terms in the second fundamental form,
$P=P_{ab}(L_{c_1c_2},R_{d_1d_2d_3d_4}^M)$.

We now return to Equation~(\ref{eqn-2.b}). We consider {symmetric
2-tensor valued} monomials of the form:
$$A=L_{a_1b_1}...L_{a_kb_k}g_{c_1d_1/e_1f_1}...g_{c_\ell d_\ell/e_\ell f_\ell}
e^u\circ e^v$$
where $k+2\ell=m-2$. We say that $A$ is
{\it admissible} if each index $2\le a\le m$ appears exactly twice
in $A$. Let $\mathcal{A}$ be the space of admissible monomials. Let
$c(A,P)$ be the coefficient of $A$ in $P$. We may then express
$$P=\sum_{A\in\mathcal{A}}c(A,P)\cdot A\,.$$
Clearly $P$ is invariant under the action of the subgroup of the
orthogonal group fixing {the normal vector $e_1$}; we let
$\mathcal{J}_{p,q}$ be the subspace of all such polynomials. Let
$\tilde m:=m-1$ be the dimension of the boundary. We will show
presently in Lemma~\ref{lem-4.2} that
\begin{equation}\label{eqn-2.c}
\dim\{\mathcal{J}_{p,q}\}=\left\{\begin{array}{lll}
1+\frac{\tilde m-1}2&\text{if }\tilde m\text{ is  odd}\\
1+\frac{\tilde m}2&\text{ if }\tilde m\text{ is even}\vphantom{\vrule height 11pt}
\end{array}\right\}\,.
\end{equation}
On the other hand the invariants
$\mathcal{F}_{m,m-2,\nu}^{(p,q),2,\partial M}$ give rise to
admissible polynomials and there are exactly this many of them. To
complete the proof, we must establish linear independence. This may
be done as follows. We have $n=m-2$. Take $(M,g)=(N^{m-1}\times
S^1,g_N\pm d\theta^2)$. It is then immediate
that
\begin{equation}
\label{eqn-2.d} g_M(\{\mathcal{F}_{m,m-2,\nu}^{(p,q),2,\partial
M}(g_M)\},\pm d\theta^2)= F_{m-1,m-2,\nu}^{(p,q),\partial N}(g_N)\,.
\end{equation}
 Since the invariants
$\{F_{m-1,m-2,\nu}\}$ are linearly independent local formulas, the
desired result follows.
\end{proof}

\section{The proof of Theorem~\ref{thm-1.4}}\label{sect-3}

\subsection{Euler Lagrange equations for a manifold with boundary}
We may integrate by parts to compute:
\begin{eqnarray*}
&&\partial_t\left.\left\{\int_M
E_{m,n}^{(p,q)}(g_t)dx_{g_t}
+\int_{\partial M}F_{m,n-1}^{(p,q),\partial M}(g_t)
dy_{g_t}\right\}\right |_{t=0}\\
&&\quad=\int_Mh_{ij}{Q^{(p,q),2}_{{m, n, ij}}}dx_g
+\sum_{k=0}^{n-1}\int_{\partial M}(\nabla_{e_1}^kh_{ij})
Q_{m,n-1-k,ij}^{(p,q),2,\partial M}dy_g\,.
\end{eqnarray*}
We examined $Q^{(p,q),2}_{m, n, ij}$ previously in Section ~\ref{sect-1.2}. Suppose first that $n=m-2$. If we consider
a product manifold of the form $(M,g_M)=(N\times S^1,g_N\pm
d\theta^2)$ and take the perturbation $h=h_N+0$, then the
Gauss-Bonnet theorem shows that the Euler-Lagrange equations are
trivial. Consequently
$$
r(Q^{(p,q),2}_{m, n, ij})=0\quad\text{and}\quad
r(Q_{m,m-2-k,ij}^{(p,q),2,\partial M})=0\,.
$$
Consequently,
by Theorem~\ref{thm-2.4},
$Q_{m,m-2-k,ij}^{(p,q),2,\partial M}=0$ for $k>0$ while there are universal
constants $d_{m,\nu}$ so that
$$Q_{m,m-2}^{(p,q),2,\partial M}=
\sum_\nu d_{m,\nu} \mathcal{F}_{m,m-2,\nu}^{(p,q),2,\partial
M}\,.$$ The precise normalizing constants can then be evaluated and
shown to be $\frac12$ by applying Equation~(\ref{eqn-2.d}) to the example
$$(M,g,h):=(N\times S^1,g_N\pm d\theta^2,\pm d\theta^2)\,.$$

 The point being that one uses the Gauss-Bonnet
theorem and notes that the variation of the volume element
$\partial_t dx_{g_t}=\frac12 dx_g$. This completes the proof of
Theorem~\ref{thm-1.4} if $n=m-2$.

Next suppose $n=m-3$. We consider
$$Q_{m,m-3}^{(p,q),2,\partial M}
-\frac12\sum_\nu  \mathcal{F}^{(p,q),2,\partial
M}_{{m,m-3},\nu}\,.$$
We use the case $n=m-2$ already established
to see this vanishes under $r_\pm$ and, hence, by
Theorem~\ref{thm-2.4}, this invariant vanishes. This completes the
proof if $n=m-3$. The general case now follows by induction. This
completes the proof of Theorem~\ref{thm-1.4}.\hfill\qed

\section{The algebraic context}\label{sect-4}
Section~\ref{sect-4} is devoted to establishing the estimate of
Equation~(\ref{eqn-2.c}) which was used in the proof of Theorem~\ref{thm-2.4}.
We work in a purely algebraic
context. In Section~\ref{sect-4.1}, we introduce the basic algebraic formalism.
 In Section~\ref{sect-4.2}, we define the relevant structure
groups. In Section~\ref{sect-4.3}, we prove the
Exchange Lemma -- this is a lemma related to orthogonal invariance.
In Section~\ref{sect-4.4}, we complete our discussion by deriving the estimate of Equation~(\ref{eqn-2.b}).

\subsection{Notational conventions}\label{sect-4.1}
Let $(V,\tilde g)$ be an inner product space of signature $(\tilde
p,\tilde q)$ and dimension $\tilde m:=\tilde p+\tilde q$; to relate
the discussion in this section to the results needed in
Section~\ref{sect-2}, we need only set $\tilde m=m-1$ and take
$V=T_Y(\partial M)$ and $\tilde g=g|_V$. We change notation slightly
and let all indices range from $1$ to $\tilde m$ rather than from
$2$ to $m=\tilde m+1$.  Let $\{e_a\}$ be an orthonormal basis for
$V$. Let $\tilde L_{ab}$ and $\tilde g_{ab/cd}$ be formal variables
where we impose the symmetries:
\begin{equation}\label{eqn-4.a}
\tilde L_{ab}=\tilde L_{ba},\quad
   \tilde g_{ab/cd}=\tilde g_{ba/cd}=\tilde g_{ab/dc}\,.
\end{equation}
Let
\begin{eqnarray*}
&&\tilde L:=\tilde L_{ab}e^a\circ e^b\in S^2(V^*),\\
&&\tilde D^2\tilde g:=\tilde g_{ab/cd}(e^a\circ
e^b)\otimes(e^c\circ e^d)\in S^2(V^*)\otimes S^2(V^*)\,.
\end{eqnarray*}
If $T$ is a linear transformation of $V$, then the natural
action of $T$ on $S^2(V^*)$ defines the action of $T$ on these variables. More
precisely, if $Te_a=T_a^be_b$, then
$$(T\tilde L)_{ab}=T_{a}^{\tilde a}T_{b}^{\tilde b}{\tilde L}_{\tilde a\tilde b}\quad\text{and}\quad
    (T\tilde D^2{\tilde g})_{ab/cd}=T_{a}^{\tilde a}T_{b}^{\tilde b}
    T_{c}^{\tilde c}T_{d}^{\tilde d}
        \tilde g_{\tilde a\tilde b/\tilde c\tilde d}\,.$$
In other words, we simply expand $\tilde L$ and $\tilde D^2{\tilde g}$ multi-linearly. This is, of course, exactly the usual
action of the general linear group on the second fundamental form and on the 2-jets of the
metric.

Consider a symmetric $2$-tensor valued monomial of the form:
$$A=\tilde L_{a_1b_1}...\tilde L_{a_kb_k}\tilde g_{c_1d_1/e_1f_1}...
\tilde g_{c_\ell d_\ell/e_\ell f_\ell} e^u\circ e^v\,.$$
We say an index $a$ {\it touches itself} in $A$ if $A$
is divisible by one of the variables $\tilde L_{aa}$, $\tilde g_{aa/\star\star}$,
$\tilde g_{\star\star/aa}$, or $e^a\circ e^a$.
Let $\delta$ be the Kronecker symbol. We let
\begin{eqnarray*}
&\operatorname{deg}_w(A):=&
\sum_{\mu=1}^k\{\delta_{a_\mu,w}+\delta_{b_\mu,w}\}
+\sum_{\nu=1}^\ell\{\delta_{c_\nu,w}+\delta_{d_\nu,w}+\delta_{e_\nu,w}+
\delta_{f_\nu,w}\}\\&&+\delta_{u,w}+\delta_{v,w},\\
&\operatorname{ord}_L(A):=&k,\quad\operatorname{ord}_{\tilde g}(A)=2\ell\,
\quad\operatorname{ord}(A):=k+2\ell,;
\end{eqnarray*}
$\operatorname{deg}_w(A)$ is the number of times that the index $w$ appears
in $A$. Motivated by the discussion of the proof of
 Theorem~\ref{thm-2.4}, we say that $A$ is
{\it admissible} if
$$\operatorname{deg}_a(A)=2\quad\text{for}\quad1\le a\le\tilde m\,.$$
Let $\mathcal{A}$ be the set of admissible monomials. If $A\in\mathcal{A}$,
then $1+k+2\ell=\tilde m$ so:
$$
\operatorname{ord}(A)=\operatorname{ord}_L(A)
+\operatorname{ord}_{\tilde g}(A)=k+2\ell=\tilde m-1\,.$$
Let $\mathcal{C}:=\{c(A,P)\}_{A\in\mathcal{A}}$ be a collection of constants.
We form the associated {\it admissible polynomial}:
$$P=P_{\mathcal{C}}:=\sum_{A\in\mathcal{A}}c(A,P)\cdot A\,.$$
We say that $A$ is a {\it monomial of $P$} if $c(A,P)\ne0$.
We may expand an admissible polynomial $P$ in the form:
\begin{equation}\label{eqn-4.b}
P=\sum_{k\equiv\tilde m-1\ (2)}P_k\quad\text{where}\quad
P_k=P_{k,\mathcal{C}}:=\sum_{A\in\mathcal{A},\operatorname{ord}_L(A)=k}
c(A,P)\cdot A\,.
\end{equation}
The symmetries
of Equation~(\ref{eqn-4.a}) mean that we can regard
$$P_k\in\otimes^{k+2\ell+1}S^2(V^*)\,.$$

\subsection{Structure groups}\label{sect-4.2}
Let $\operatorname{GL}$ be the general linear group of $V$ and
let $\mathcal{O}$ be the associated orthogonal group:
$$
\mathcal{O}:=\{T\in\operatorname{GL}:g(Tx,Ty)=g(x,y)
     \text{ for all }x,y\in V\, \}.
$$
For $1\le k\le\tilde m$, let
$$\mathcal{O}_k:=\{T\in\mathcal{O}:Te_b=e_b\text{ for all }b<k\}\,.
$$
If $\{a,b\}$ are distinct indices, let $V_{a,b}:=\operatorname{Span}\{e_a,e_b\}$.
If $V_{a,b}$ has signature $(2,0)$ or has signature $(0,2)$, we consider the rotations $T_{a,b}(\theta)\in\mathcal{O}$:
defined by setting:
$$T_{a,b}(\theta)e_c:=\left\{\begin{array}{rll}
\cos\theta e_a+\sin\theta e_b&\text{if}&c=a\\
-\sin\theta e_a+\cos\theta e_b&\text{if}&c=b\\
e_c&\text{if}&c\ne a\text{ and }c\ne b\end{array}\right\}\,.
$$
Similarly, if $V_{a,b}$ has signature $(1,1)$, we
consider the hyperbolic boosts $T_{a,b}(\theta)\in\mathcal{O}$
defined by setting:
$$T_{a,b}(\theta)e_c:=\left\{\begin{array}{rll}
\cosh\theta e_a+\sinh\theta e_b&\text{if}&c=a\\
\sinh\theta e_a+\cosh\theta e_b&\text{if}&c=b\\
e_c&\text{if}&c\ne a\text{ and }c\ne b
\end{array}\right\}\,.
$$
The transformations $\{T_{a,b}(\theta)\}$ for $\theta\in\mathbb{R}$ and $1\le a<b\le\tilde m$
generate the connected component of the identity of $\mathcal{O}$.

Each index appears exactly twice in any admissible monomial. If $V_{a,b}$ has signature $(1,1)$, then we replace
 the two `$a$' indices by `$\cosh\theta
a+\sinh\theta b$', we replace the two `$b$' indices by `$\sinh\theta a+
\cosh\theta b$', and we expand multi-linearly to determine $T_{a,b}(\theta)A$; if $V_{a,b}$ has signature
$(2,0)$ or $(0,2)$, then there are sign changes. We replace the two `$a$' indices by $\cos\theta
a+\sin\theta b$ and we replace the two `$b$' indices by $-\sin\theta
a+\cos\theta b$ before expanding multi-linearly. Thus each admissible
monomial gives rise to 16 different monomials which must be combined
and simplified in computing the action of $T_{a,b}(\theta)$ on an
admissible polynomial. We do not sum over repeated induces in what
follows in the remainder of Section~\ref{sect-4.2}. If $V_{a,b}$ has signature
$(1,1)$, we have:
\begin{eqnarray*}
T_{a,b}(\theta)\tilde g_{aa/bb}&=&
\cosh^4\theta\tilde g_{aa/bb}
+\cosh^3\theta\sinh\theta(2\tilde g_{ab/bb}+2\tilde g_{aa/ab})\\
&+&\cosh^2\theta\sinh^2\theta(\tilde g_{aa/aa}+\tilde g_{bb/bb}
+ 4\tilde g_{ab/ab})\\
&+&\cosh\theta\sinh^3\theta(2\tilde g_{ab/aa}+2\tilde g_{bb/ab})
+\sinh^4\theta\tilde g_{bb/aa}\,.
\end{eqnarray*}
The 4 terms (counted with multiplicity) with a coefficient of $\cosh^3\theta\sinh\theta$
arise from changing a single index $a\rightarrow b$ or $b\rightarrow a$; if $V_{a,b}$ has
signature $(2,0)$ or $(0,2)$, then there are
appropriate changes of sign:
\begin{eqnarray*}
T_{a,b}(\theta)\tilde g_{aa/bb}&=&
\cos^4\theta\tilde g_{aa/bb}
+\cos^3\theta\sin\theta(2\tilde g_{ab/bb}-2\tilde g_{aa/ab})\\
&+&\cos^2\theta\sin^2\theta(\tilde g_{aa/aa}+\tilde g_{bb/bb}
- 4\tilde g_{ab/ab})\\
&+&\cos\theta\sin^3\theta(2\tilde g_{ab/aa}-2\tilde g_{bb/ab})
+\sin^4\theta\tilde g_{bb/aa}\,.
\end{eqnarray*}
Suppose that $\tilde m=2$ {so indices range from $1$ to $2$.} Let
$$P=\det(L)=\tilde L_{11}\tilde L_{22}-\tilde L_{12}\tilde L_{12}\,.$$
Assume $V_{1,2}$ has signature $(1,1)$. Then:
\begin{eqnarray*}
&&T_{12}(\theta)\tilde L_{11} = \cosh^2\theta \tilde L_{11}+
\sinh^2\theta \tilde L_{22} +2\cosh\theta\sinh\theta \tilde
L_{12},\\
&&T_{12}(\theta)\tilde L_{12} = (\cosh^2\theta + \sinh^2\theta)
\tilde L_{12}+\cosh\theta\sinh\theta (\tilde L_{11} + \tilde
L_{22}),\\
&&T_{12}(\theta)\tilde L_{22} = \sinh^2\theta \tilde
L_{11}+ \cosh^2\theta \tilde L_{22} +2\cosh\theta\sinh\theta \tilde
L_{12}\,.
\end{eqnarray*}
We compute:
\begin{eqnarray*}
&&T_{12}\tilde L_{11}\tilde L_{22}=\cosh^4\theta \tilde L_{11}\tilde L_{22}
+\cosh^3\theta\sinh\theta(2\tilde L_{12}\tilde L_{22}+2\tilde L_{11}\tilde L_{12})\\
&&\quad+\cosh^2\theta\sinh^2\theta(\tilde L_{11} \tilde L_{11}+\tilde L_{22}\tilde L_{22}+4\tilde L_{12}\tilde L_{12})\\
&&\quad+\cosh\theta\sinh^3\theta(2\tilde L_{12}\tilde L_{22}+2\tilde
L_{11}\tilde L_{12})
+\sinh^4\theta\tilde L_{11}\tilde L_{22},\\
&&T_{12}\tilde L_{12}\tilde L_{12}=\cosh^4\theta \tilde L_{12}\tilde L_{12}
+\cosh^3\theta\sinh\theta(2\tilde L_{12}\tilde L_{22}+2\tilde L_{11}\tilde L_{12})\\
&&\quad+\cosh^2\theta\sinh^2\theta(2\tilde L_{12}\tilde
L_{12}+\tilde L_{11}\tilde L_{11}
+\tilde L_{22}\tilde L_{22}+2\tilde L_{11}\tilde L_{22})\\
&&\quad +\cosh\theta\sinh^3\theta(2\tilde L_{11}\tilde
L_{12}+2\tilde L_{22}\tilde L_{12})
+\sinh^4\theta\tilde L_{12}\tilde L_{12}\\
&&T_{12}(\tilde L_{11}\tilde L_{22}-\tilde L_{12}\tilde L_{12})=\cosh^4\theta(\tilde L_{11}\tilde L_{22}-\tilde L_{12}\tilde L_{12})\\
&&\quad-2\cosh^2\theta\sinh^2\theta(\tilde L_{11}\tilde L_{22}-\tilde L_{12}\tilde L_{12})
+\sinh^4\theta(\tilde L_{11}\tilde L_{22}-\tilde L_{12}\tilde L_{12})\\
&&\quad=\tilde L_{11}\tilde L_{22}-\tilde L_{12}\tilde L_{12}\,.
\end{eqnarray*}
This shows, not surprisingly, that $\det(L)$ is invariant under the action
of the orthogonal group in dimension 2.

\subsection{The exchange lemma}\label{sect-4.3}
Suppose given a monomial $C$ with
$\operatorname{deg}_a(C)=3$, $\operatorname{deg}_b(C)=1$, and
$\operatorname{deg}_c(C)=2$ for $c\ne a,b$. We suppose that the index
$a$ touches itself in $C$. There are then exactly $2$ admissible monomials $A$ and $B$ which transform
to $C$ by changing a single index $a\rightarrow b$; there are no admissible
monomials which transform to $C$ by changing a single index $b\rightarrow a$
since such a monomial would have the index $a$ appearing 4 times.
There are several possibilities:
\begin{enumerate}
\item If $C=\tilde L_{aa}A_1$, then $A=\tilde L_{ab}A_1$ and
$B=\tilde L_{aa}B_1$ where to define $B_1$,
the index $a$ appearing in $A_1$ is exchanged for the index $b$.
\item If $C=\tilde g_{aa/cd}A_1$ (resp. $\tilde g_{cd/aa}A_1$), then $A=\tilde g_{ab/cd}A_1$ (resp. $\tilde g_{cd/ab}A_1$) and
$B=\tilde g_{aa/\tilde c\tilde d}B_1$ (resp. $\tilde g_{\tilde c\tilde d/aa}B_1$) where to define $B_1$, the index $a$ appearing in $(cd)A_1$
 is exchanged for the index $b$.
\item If $C=A_1e^a\circ e^a$, then $A=A_1e^a\circ e^b$ and
$B=B_1e^a\circ e^a$ where to define $B_1$, the index $a$ appearing in $A_1$ is
exchanged for the index $b$. \end{enumerate}

The following technical Lemma appeared first in the discussion of
\cite{G73} of the heat equation proof of the Gauss-Bonnet
formula when Theorem~\ref{thm-2.3} was first proved; we modify the
proof given there to make it applicable to the present context.
\begin{lemma}
Let $P=P_{\mathcal{C}}$ be an admissible polynomial and let
$\{a,b\}$ be distinct indices. Assume that $T_{a,b}(\theta)P=P$ for all $\theta$.
Let $\{C,A,B\}$ be as above. Then $A$ is a monomial of $P$ if and only if $B$ is a monomial of $P$.
\end{lemma}

\begin{proof}
Suppose first that $V_{a,b}$ has signature $(1,1)$.
We expand $T_{a,b}(\theta)A$ by replacing each index $a$ by  $\cosh\theta a+\sinh\theta b$
and each index $b$ by $\sinh\theta a+\cosh\theta b$ and expanding multi-linearly. To obtain $C$, we must change
an odd number of indices.
 We consider the coefficient of $\cosh^3\theta\sinh\theta$ in $c(C,T_{a,b}(\theta)A)$; this arises from changing exactly one index and
 leaving the other three indices the same. Since $\operatorname{deg}_b(C)=1$ and $\operatorname{deg}_b(A)=2$,
 the index that is being changed is $b\rightarrow a$. This may, of course, be done either in one way
 (in which case we set $\sigma_1=1$)
or two ways (in which case we set $\sigma_1=2$).  To illustrate this, we consider
the case (1) above. For example, if we have that $A=\tilde L_{aa}\tilde L_{bc}\tilde
L_{db}A_2$ for $\{a,b,c,d\}$ distinct indices, then $\sigma_1=1$
while if $A=\tilde L_{aa}\tilde L_{cb}\tilde L_{cb}A_2$ for
$\{a,b,c\}$ distinct indices, then $\sigma_1=2$. {The arguments for the
other possible cases (2) and (3) are essentially similar.

The argument given above shows that
$$c(C,T_{a,b}(\theta)A)=\sigma_1\cosh^3\theta\sinh\theta+\star\cosh\theta\sinh^3\theta\text{ for }\sigma_1\in\{1,2\}$$
and where $\star$ is a coefficient which is not of interest.
Similarly
$$c(C,T_{a,b}(\theta)B)=\sigma_2\cosh^3\theta\sinh\theta+\star\cosh\theta\sinh^3\theta\text{ for }\sigma_2\in\{1,2\}\,.$$

Suppose that $X$ is an admissible monomial so that
$X$ transforms to $C$ by changing exactly one index $a\rightarrow b$ or $b\rightarrow a$. Since $\operatorname{deg}_a(X)=2$
and $\operatorname{deg}_a(C)=3$, $X$ can not transform to $C$ by changing the index $a$ to $b$ and thus transforms to $C$
by changing the index $b$ to $a$; conversely $X$ is obtained from $C$ by changing exactly one index $a$ to $b$ and leaving
the other indices alone. Consequently $X=A$ or $X=B$ since these are the only monomials obtained in this way. Since the
number of indices which must be changed is odd, the powers of $\cosh$ and $\sinh$ are odd and we have:
$$
0=c(C,T_{a,b}(\theta)P)=
(\sigma_1c(A,P)+\sigma_2c(B,P))\cosh^3\theta\sinh\theta+\star\cosh\theta\sinh^3\theta\,.
$$
We eliminate the coefficient of $\cosh\theta\sinh^3\theta$ by examining:
$$
0=\lim_{\theta\rightarrow0}
\frac{c(C,T_{a,b}(\theta)P)}{\sinh\theta}=\sigma_1c(A,P)+\sigma_2c(B,P)\,.
$$
Since $\sigma_i\in\{1,2\}$, $c(A,P)$ is non-zero if and only if $c(B,P)$ is non-zero which proves the Lemma if
$V_{a,b}$ has signature $(1,1)$. If $V_{a,b}$ has signature $(0,2)$ or $(2,0)$, then we have similarly that:
$$
0=-(\sigma_1c(A,P)+\sigma_2c(B,P))\cos^3\theta\sin\theta+\star\cos\theta\sin^3\theta
$$
and the argument again is similar.}
\end{proof}

\subsection{The estimate of Equation~(\ref{eqn-2.b})}\label{sect-4.4}
Let $\mathcal{J}$ be the space of admissible polynomials which are
invariant under the action of the orthogonal group $\mathcal{O}$. For
$0\le k\le\tilde m-1$ and $k\equiv\tilde m-1$ mod $2$, set
$$Q_k:=\tilde L_{a_1b_1}...\tilde L_{a_kb_k}
\tilde g_{a_{k+1}b_{k+1}/a_{k+2}b_{k+2}}
...\tilde g_{a_{\tilde m-2}b_{\tilde m-2}/a_{\tilde m-1}b_{\tilde m-1}}
e^{a_{\tilde m}}\circ e^{b_{\tilde m}}
\varepsilon_{b_1...b_{\tilde m}}^{a_1...a_{\tilde m}}\,.$$
\begin{lemma}\label{lem-4.2}
The polynomials $Q_k$ for $k\equiv\tilde m-1$ mod $2$
are a basis for $\mathcal{J}$. Thus
$$\dim\{\mathcal{J}\}=\left\{\begin{array}{lll}
1+\frac{\tilde m-1}2&\text{if }\tilde m\text{ is  odd}\\
1+\frac{\tilde m}2&\text{if }\tilde m\text{ is even}\vphantom{\vrule height 11pt}
\end{array}\right\}\,.
$$
\end{lemma}

\begin{proof} Let $P\in\mathcal{J}$. Adopt the notation of Equation~(\ref{eqn-4.b})
to express $P=\sum_kP_k$.
Then the $P_k$ are each invariant under the action of $\mathcal{O}$ separately.
Let $\mathcal{J}_k$ be the span of such polynomials;
$\mathcal{J}=\oplus_k\mathcal{J}_k$. We will
complete the proof by showing $\dim\{\mathcal{J}_k\}\le 1$.
Let $0\ne P_k\in\mathcal{J}_k$. Let
$$A=\tilde L_{a_1b_1}...\tilde L_{a_kb_k}
\tilde g_{a_{k+1}b_{k+1}/a_{k+2}b_{k+2}} ...\tilde g_{a_{\tilde m-2}b_{\tilde
m-2}/a_{\tilde m-1}b_{\tilde m-1}} e^{a_{\tilde m}}\circ
e^{b_{\tilde m}}$$ be a monomial of $P$. We can apply the exchange
Lemma to assume $a_1=b_1$. If $a_1=1$, fine. Otherwise, we can apply
the exchange Lemma to assume $b_1=1$ and then apply the exchange
Lemma again to construct a monomial $A_1$ of $P$ so that
$a_1=b_1=1$. Decompose $$P_k=\tilde L_{11}P_{k,1}+\tilde P_{k,1}$$
where $\tilde P_{k,1}:=P_k-\tilde L_{11}P_{k,1}$ and where
$$
P_{k,1}=\frac1{\tilde L_{11}}\sum_{B\in\mathcal{A}:\tilde L_{11}\text{ divides }B}c(B,P)\cdot B\,.
$$
Then $0\ne P_{k,1}$ and $P_{k,1}$ is invariant under the action $\mathcal{O}_2$
since we have fixed the index `$1$'. Thus, in particular, $P_{k,1}$
is invariant under the action of $T_{a,b}(\theta)$ for $2\le a<b\le\tilde m$.
We can then apply the exchange Lemma
to show that $\tilde L_{22}$ divides some monomial of $P_{k,1}$. We continue
in this fashion to construct an admissible polynomial $0\ne P_{k,k}$ which
is invariant under the action $\mathcal{O}_{k+1}$ and
so that $P_{k,k}$ is divisible by $\tilde L_{11}...\tilde L_{kk}$. We express
$$P_{k,k}=\tilde L_{11}...\tilde L_{kk}\tilde g_{a_{k+1}b_{k+1}/a_{k+2}b_{k+2}}...
\tilde g_{a_{\tilde m-2}b_{\tilde m-2}/a_{\tilde m-1}b_{\tilde m-1}}e^{a_{\tilde m}}
\circ e^{b_{\tilde m}}\,.$$
Of course, if $k=0$, then $P_{k,k}=P$ where as if $k=\tilde m-1$, then
we shall not proceed further.

If $k<\tilde m-1$, we apply the exchange Lemma twice to choose
a monomial so $a_{k+1}=b_{k+1}=k+1$ and continue in this fashion
finally to show that
$$\tilde L_{11}...\tilde L_{kk}\tilde g_{k+1,k+1/k+2,k+2}...
\tilde g_{\tilde m-2,\tilde m-2/\tilde m-1,\tilde m-1}e^{a_{\tilde m}}\circ
e^{b^{\tilde m}}$$ is a monomial of $P$. Since the index
$\tilde m$ can only appear in $\{a_{\tilde m},b_{\tilde m}\}$ we
have $a_{\tilde m}=b_{\tilde m}=\tilde m$. We conclude therefore
that
$$A_k:=\tilde L_{11}...\tilde L_{kk}\tilde g_{k+1,k+1/k+2,k+2}...
\tilde g_{\tilde m-2,\tilde m-2/\tilde m-1,\tilde m-1}e^{\tilde m}\circ e^{\tilde m}$$
is a monomial of $P_k$. We summarize. If $0\ne P_k\in\mathcal{J}_k$, then
$c(A_k,P_k)\ne0$. This shows that $\dim\{\mathcal{J}_k\}\le1$.\end{proof}

\begin{remark}\rm
 H. Weyl's second theorem of invariants states that all relations
amongst orthogonal linear invariants arising from contractions of
indices can be constructed using the tensor  $\varepsilon$ described
in Equation~(\ref{eqn-1.a}). When passing from the case of general
orthogonal linear invariants to the case of universal curvature
invariants, one must also include the curvature symmetries of
Equation~(\ref{eqn-2.a}). This plays a crucial role in the proof of
Theorem~\ref{thm-2.3} given in \cite{GPS11,GPS12}. However, rather
than using H. Weyl's theorems directly (which would necessitate
extending the discussion to include covariant derivatives of the
second fundamental form), we have chosen to give a proof of
Theorem~\ref{thm-1.4} based on the analysis of \cite{G73}, which was
first developed to give a heat equation proof of the
Gauss-Bonnet theorem.
\end{remark}

\section*{Acknowledgments}
Research of P. Gilkey partially supported by project MTM2009-07756
(Spain).  Research of {J.H.} Park and K. Sekigawa was supported by
the National Research Foundation of Korea (NRF) grant funded by the
Korea government (MEST) (2012-0005282).

\end{document}